\documentclass[11pt]{amsart}

\usepackage{a4wide,amsmath,amssymb}
\usepackage[toc,page]{appendix}

\usepackage{url}

\usepackage{hyperref}

\usepackage{mathtools}

\usepackage{xcolor}

\newtheorem{thm}{Theorem}[section]
\newtheorem{lemma}[thm]{Lemma}

\theoremstyle{remark}
\newtheorem{rem}[thm]{Remark}

\theoremstyle{definition}

\newtheoremstyle{Claim}{}{}{\itshape}{}{\itshape\bfseries}{:}{ }{#1}
\theoremstyle{Claim}

\newcommand{\T}{{\mathbb{T}}}

\newcommand{\R}{\mathbb{R}}

\newcommand{\N}{\mathbb{N}}

\newcommand{\eps}{\varepsilon}

\makeatletter
\@namedef{subjclassname@2020}{\textup{2020} Mathematics Subject Classification}
\makeatother

\theoremstyle{plain}

\begin{document}

\title[]{Remarks on the rate of convergence of the vanishing viscosity process of Hamilton-Jacobi equations}

\author{Alessandro Goffi}
\address{Dipartimento di Matematica e Informatica ``Ulisse Dini'', Universit\`a degli Studi di Firenze, 
viale G. Morgagni 67/A, 50134 Firenze (Italy)}
\curraddr{}
\email{alessandro.goffi@unifi.it}

\thanks{The author is member of the Gruppo Nazionale per l'Analisi Matematica, la Probabilit\`a e le loro Applicazioni (GNAMPA) of the Istituto Nazionale di Alta Matematica (INdAM), and he was partially supported by the INdAM-GNAMPA project 2024.
}

\date{\today}


\begin{abstract}
We establish a linear $L^p$ rate of convergence, $1<p<\infty$, with respect to the viscosity $\eps$ for the vanishing viscosity process of semiconcave solutions of Hamilton-Jacobi equations by regularizing the PDE with the half-Laplacian $-\eps(-\Delta)^{1/2}$. Our result reveals a nonlocal phenomenon, since it improves the known estimates obtained via the classical second order vanishing viscosity regularization $\eps\Delta u$. It also highlights a faster rate of convergence than the available $\mathcal{O}(\eps|\log\eps|)$ rate in sup-norm obtained by the doubling of variable technique for this nonlocal approximation. The result is based on integral methods and does not use the maximum principle.
\end{abstract}

\maketitle

\section{Introduction}
It is well-known that viscosity solutions $u:\R^n\times(0,T)\to\R$ of the first-order Hamilton-Jacobi equation
\begin{equation}\label{HJintro}
\begin{cases}
\partial_t u+H(Du)=f(x,t)&\text{ in }\R^n\times(0,T),\\
u(x,0)=u_0(x)&\text{ in }\R^n,
\end{cases}
\end{equation}
can be obtained by the so-called method of vanishing viscosity \cite{BCD,CL83tams,CEL}, i.e. taking the limit as $\eps\to0^+$ of the solution $u_\eps$ of the parabolic regularization
\begin{equation}\label{HJintroVisc}
\begin{cases}
\partial_t u_\eps-\eps\Delta u_\eps+H(Du_\eps)=f(x,t)&\text{ in }\R^n\times(0,T),\\
u_\eps(x,0)=u_0(x)&\text{ in }\R^n.
\end{cases}
\end{equation}
This method, going back to the theory of scalar conservation laws in fluid dynamics, allows to prove, under certain assumptions on $H:\R^n\to\R$ and the data of the equation, the uniform convergence of $u_\eps$ to a solution $u$ of \eqref{HJintro} on compact subsets of $\R^n\times(0,T)$, see \cite{BCD,CEL,CL83tams,L82book}. The rate of the convergence of this process depends strongly on the interplay between the hypotheses of $H$, $u_0$ and $u$, and it has been the object of intensive research. The interest on this problem has been renewed recently in the theory of Mean Field Control and Games, see for instance \cite[Section 3]{Cecchin} and \cite{DaudinDelarueJackson}.\\

 When $H$ is locally Lipschitz and $u\in W^{1,\infty}_x$, which in turn requires certain additional coercivity conditions on $H$, one proves the $\mathcal{O}(\sqrt{\eps})$ rate \cite{CL84,EvansARMA,K67II,L82book,Souganidis}, see also \cite{Calder},
\begin{equation}\label{root}
\|u_\eps-u\|_{L^\infty(\R^n\times(0,T))}\leq C\sqrt{\eps}.
\end{equation}
This is in general optimal in view of the counterexample recently found in \cite{Tranetal}, and the constant $C$ can be made explicit, behaving as $\mathcal{O}(\sqrt{T\eps})$. When the data are less regular, e.g. $H$ and/or $u$ are H\"older continuous, the rate of convergence becomes slower, as shown in \cite{BCD} via the doubling of variable technique.\\
When $u$ is semiconcave or $H$ is uniformly convex, one expects a better rate of convergence: in fact, when $\Delta u\leq C$ one proves the one-side linear rate
\[
u_\eps-u\leq C\eps,
\]
see e.g. \cite{Calder,CGM,L82book}, or an average two-side $\mathcal{O}(\eps)$ rate \cite{TranBook} when $H$ is uniformly convex. However, a better lower bound than $\mathcal{O}(\sqrt{\eps})$ on $u_\eps-u$ remains unknown, even when $H$ is uniformly convex, except for some specific low-dimensional examples \cite{Tranetal}. The recent paper \cite{CGM}, see also the earlier analysis in \cite{LinTadmor}, established, in the case of semiconcave solutions and the flat torus $\T^n$, the $L^1$ linear rate
\[
\|u_\eps-u\|_{L^\infty(0,T;L^1(\T^n))}\leq C\eps,
\]
valid also for some nonconvex Hamiltonians $H$. As a byproduct, this gave a rate of convergence for gradients in the vanishing viscosity process and a rate of convergence for certain classes of quasilinear hyperbolic systems. Notably, it provides by interpolation with \eqref{root} the estimate
\[
\|u_\eps-u\|_{L^\infty(0,T;L^p(\T^n))}\leq C\eps^{\frac12+\frac{1}{2p}}.
\]
This bound deteriorates as $p\to+\infty$ to the classical $\mathcal{O}(\sqrt{\eps})$ rate for Lipschitz solutions. The aim of this note is to consider the rate of convergence for semiconcave solutions of \eqref{HJintro} by exploiting a different regularization with the nonlocal operator $\eps(-\Delta)^s$, $s\in(0,1)$. Therefore, we consider the parabolic nonlocal problem for periodic solutions $u=u_{\eps,s}$ of 
\begin{equation}\label{HJintroViscFrac}
\begin{cases}
\partial_t u_{\eps,s}+\eps(-\Delta )^su_{\eps,s}+H(Du_{\eps,s})=f(x,t)&\text{ in }\R^n\times(0,T),\\
u_{\eps,s}(x,0)=u_0(x)&\text{ in }\R^n.
\end{cases}
\end{equation}
Here $(-\Delta)^s=(-\Delta_{\T^n})^s$ is the fractional Laplacian on the flat torus $\T^n$: this can be defined via multiple Fourier series and, for suitably regular functions, one has the standard representation formula, see \cite{RoncalStinga}. Nonetheless, we will never make use of integral formulas in the sequel and this is a distinctive feature of our method of proof compared to the literature \cite{Biswas,Droniou,DroniouImbert}. It is well-known that solutions of first-order Hamilton-Jacobi equations with uniformly convex Hamiltonians are semiconcave, regardless of the regularity of the initial datum \cite{K67II,EvansARMA}: this property implies the uniqueness of Lipschitz solutions for Hamilton-Jacobi equations with convex $H$, and continues to  hold for its nonlocal counterpart, being the estimate independent of the viscosity, cf. \cite{CGsima}. \\
The study of this nonlocal approximation and the related speed of convergence date back to \cite{Biswas,DroniouImbert} and earlier to the study of nonlocal conservation laws \cite{Droniou}. In the case of Lipschitz solutions and locally Lipschitz Hamiltonians, the authors in \cite{Biswas,DroniouImbert} obtained the following rates for the nonlocal vanishing viscosity process in sup-norm
\[
\|u_{\eps,s}-u\|_{\infty}=
\begin{cases}
\mathcal{O}(\eps)&\text{ if }0<s<\frac12,\\
\mathcal{O}(\eps|\log\eps|)&\text{ if }s=\frac12,\\
\mathcal{O}(\eps^\frac{1}{2s})&\text{ if }\frac12<s<1.
\end{cases}
\]
In particular, we show here that, in the borderline case $s=\frac12$, this rate can be improved to $\mathcal{O}(\eps)$ in any $L^p$ norm for finite $p\in(1,\infty)$ by means of different methods, see Theorem \ref{main}. This reveals a nonlocal phenomenon highlighting a faster approximation via this nonlocal vanishing viscosity procedure. Remarkably, it boosts the rate of convergence of the local case $s=1$, as we show the $p$-independent linear rate
\[
\|u_{\eps,\frac12}-u\|_{L^\infty(0,T;L^p(\T^n))}\leq C\eps,\ 1<p<\infty.
\]
Nonetheless, we are not able to reach the endpoint $p=\infty$, a case which remains open in general even when $s=1$ for semiconcave solutions, see Remark \ref{finalrem} for further comments. As far as the method of proof is concerned, we will use integral duality methods as initiated by L.C. Evans \cite{EvansARMA} and exploit integrability estimates of nonlocal Fokker-Planck equations with weakly compressible velocity fields $b$, i.e. satisfying
\[
[\mathrm{div}(b)]^-\in L^1_t(L^\infty_x),
\]
when the space-time drift $b(x,t)\sim-D_pH(Du_\eps(x,t))$ (i.e. the drift of the linearization of \eqref{HJintroVisc}). Integral estimates for first-order transport and continuity equations are classical under this condition within the Diperna-Lions setting, cf. \cite{LBL}. This assumption is always satisfied when the solution $u$ of \eqref{HJintroViscFrac} is semiconcave and $H$ is uniformly convex. Our linear rate is then deduced by means of the boundedness of Riesz transform appearing in the theory of singular integrals due to A. P. Calder\'on, cf. \cite{AdamsHedberg,Stein}, namely
\[
A^{-1}\|(I-\Delta)^{\frac{\alpha}{2}}u\|_{L^p}\leq \|u\|_{W^{\alpha,p}}\leq A\|(I-\Delta)^{\frac{\alpha}{2}}u\|_{L^p},\ \alpha\in\N,\ 1<p<\infty, A>0.
\]
A motivation for this analysis is, among others, the study of the rate of convergence of nonlocal vanishing viscosity approximations of the first-order Mean Field Games system
\[
\begin{cases}
-\partial_t u+H(Du)=F[m]&\text{ in }Q_T,\\
\partial_t m-\mathrm{div}(D_pH(Du)m)=0&\text{ in }Q_T,\\
u(x,T)=u_T(x),\ m(x,0)=m_0(x)&\text{ in }\T^n.
\end{cases}
\]
We expect that the result of the manuscript could improve the known speed of convergence with respect to the classical local viscous approximation, at least in the case of regularizing couplings $F$ and under certain assumptions on $H$ and $u_T$. These aspects will be the matter of future research.
\section{Rate of convergence of the vanishing viscosity process for HJ equations via nonlocal operators}

Consider the regularized Cauchy problem satisfied by $u_{\eps,\frac12}:=u_\eps$
\begin{equation}\label{HJvisc}
\begin{cases}
\partial_t u_\eps+\eps(-\Delta)^\frac12 u_\eps+H(Du_\eps)=f(x,t)&\text{ in }Q_T:=\T^n\times(0,T),\\
u_\eps(x,0)=u_0(x)&\text{ in }\T^n.
\end{cases}
\end{equation}
Here $H:\R^n\to\R$, the so-called Hamiltonian, is a $C^2$, uniformly convex, function. We will assume that $f:Q_T\to\R$ is $L^1_t(L^\infty_x)$ semiconcave, i.e.
\[
D^2f\leq c_f(t)\in L^1_t.
\]
In place of the previous inequality we will sometimes write $\|(D^2 f(t))^+\|_{L^\infty(\T^n)}\leq c_f(t)$ or use the compact notation $\|(D^2f)^+\|_{L^1_t(L^\infty_x)}<\infty$.\\

Denote by $u$ the solution of the first-order equation
\begin{equation}\label{HJ}
\begin{cases}
\partial_t u+H(Du)=f(x,t)&\text{ in }Q_T,\\
u(x,0)=u_0(x)&\text{ in }\T^n.
\end{cases}
\end{equation}
We prove the following preliminary result: it shows $L^q$ integrability estimates for the dual of the linearization of \eqref{HJvisc} under suitable assumptions on the negative part of the divergence of its velocity field:
\begin{lemma}\label{lemmaFP}
Let $\alpha\in L^{q}(\T^n)$, $1<q<\infty$, $\alpha\geq0$ and $\rho=\rho_\eta$, $\eta\geq0$, be the nonnegative solution to the backward problem
\begin{equation}\label{fp12gen}
\begin{cases}
-\partial_t \rho+\eta(-\Delta)^\frac12 \rho+\mathrm{div}(b(x,t)\rho)=0&\text{ in }Q_\tau:=\T^n\times(0,\tau),\\
\rho(\tau)=\alpha&\text{ in }\T^n.
\end{cases}
\end{equation}
Assume that \begin{equation}\label{divb}
[\mathrm{div}(b)]^-\in L^1(0,\tau;L^\infty(\T^n)).
\end{equation}
Then there exists a constant $C_F>0$, depending on $q$ and $\|[\mathrm{div}(b)]^-\|_{L^1(0,\tau;L^\infty(\T^n))}$ but not on $\eta$, such that
\[
\|\rho(\omega)\|_{L^{q}(\T^n)}\leq C_F\|\rho(\tau)\|_{L^{q}(\T^n)},\ \omega\in[0,\tau).
\]
The estimate also holds for $q=\infty$ and the constant is stable in the limit $q\to\infty$.
\end{lemma}
\begin{proof}
We consider $\widetilde \rho(x,t)=\rho(x,\tau-t)$, so that the Cauchy problem \eqref{fp12gen} becomes forward in time. In this case one obtains the estimate by approximation, testing \eqref{fp12gen} with $q\widetilde \rho^{q-1}$, $q>1$, which leads to the following inequality
\[
\frac{d}{dt}\int_{\T^n}\widetilde\rho^q(t)\,dx\leq (q-1)\|[\mathrm{div}(b)(\tau -t)]^-\|_{L^\infty(\T^n)}\int_{\T^n}\widetilde\rho^q(t)\,dx.
\]
This implies the estimate using a Gronwall-type argument. The only difference with the local case driven by the Laplacian is the treatment of the nonlocal term. While in the local case $s=1$ one has
\[
-\eta q\int_{\T^n}\widetilde\rho^{q-1}(t)\Delta \widetilde\rho(t)\,dx=\eta\ q(q-1)\int_{\T^n}\widetilde\rho^{q-2}(t)|D\widetilde\rho(t)|^2\,dx\geq0
\]
directly by integrating by parts, here we need a lower bound for fractional derivatives. One can apply for instance Lemma 2.4 and Lemma 2.5 in \cite{Cordoba} or Lemma 1 in \cite{CordobaPNAS} (see also Appendix A in \cite{Constantin} for a proof via the integral formula) to find the inequality
\[
\eta q\int_{\T^n} \widetilde\rho^{q-1}(t)(-\Delta)^\frac12\widetilde\rho(t)\,dx\geq \eta c_q\int_{\T^n}|(-\Delta)^\frac14\widetilde\rho^{\frac{q}{2}}(t)|^2\,dx(\geq0).
\]
\end{proof}
\begin{rem}\label{energy}
The existence and uniqueness of weak solutions $\rho$ in the class
\[
Y=\{\rho\in L^2(0,\tau;H^\frac12(\T^n)),\partial_t\rho\in L^2(0,\tau;H^{-1}(\T^n))\}
\]
under \eqref{divb} can be obtained following the scheme of \cite[Theorem 4.3]{Figalli}, via an abstract theorem of J.-L. Lions, see \cite[Theorem 4.6]{Figalli}. The well-posedness for nonlocal Fokker-Planck equations is studied also in \cite{Tian} in the context of the Diperna-Lions theory under the assumption \eqref{divb}. We point out that the well-posedness in $Y$ guarantees the use of energy methods and integration by parts in energy spaces in the next proof, see Remark 3.2 in \cite{CGsima}.
\end{rem}
The next is the main result of the paper that we show in the model case of uniformly convex, smooth, Hamiltonians. It is an a priori estimate on the rate of convergence of the vanishing viscosity regularization via the half-Laplacian. We refer the reader to the next Remark \ref{comparison} for a detailed comparison with the literature. 
\begin{thm}\label{main}
Assume that $u_0:\T^n\to\R$ is semiconcave, and let $f:Q_T\to\R$ be $L^1_t(L^\infty_x)$-semiconcave  and $H\in C^2(\R^n)$ be uniformly convex, namely $\theta I_n\leq D^2_{pp}H(p)\leq \Theta I_n$, $0<\theta\leq\Theta$, $\theta,\Theta\in\R$. Then
\[
\|u_{\eps,\frac12}-u\|_{L^\infty(0,T;L^p(\T^n))}\leq C\eps,\ 1< p<\infty,
\]
where $C$ depends on $n,p,\|(D^2f)^+\|_{L^1_t(L^\infty_x)},\theta,\Theta,T,\|(D^2u_0)^+\|_{L^\infty}$, but not on $\eps$.
\end{thm}
\begin{proof}
We drop the subscript $s$ in $u_{\eps,s}$ for brevity. Set $w=u_\eps-u_\eta$, and note that by the linearity of the fractional Laplacian the function $w$ solves the PDE
\begin{equation}\label{eqw}
\partial_t w+\eta(-\Delta)^\frac12 w-b(x,t)\cdot Dw=-(\eps-\eta)(-\Delta)^\frac12 u_\eps,
\end{equation}
where the drift $b$ is given by $b(x,t)=-\int_0^1 D_pH(\zeta Du_\eps+(1-\zeta)Du_\eta)d\zeta$. We consider the dual problem $\rho:=\rho_\eta$
\begin{equation}\label{fp12}
\begin{cases}
-\partial_t \rho+\eta(-\Delta)^\frac12 \rho+\mathrm{div}(b(x,t)\rho)=0&\text{ in }Q_\tau:=\T^n\times(0,\tau),\\
\rho(\tau)=\alpha&\text{ in }\T^n,
\end{cases}
\end{equation}
where $\alpha\in L^{p'}$, $\alpha\geq0$ and $\|\alpha\|_{L^{p'}(\T^n)}=1$ (it is enough to take $\alpha(\tau)=\frac{w^{p-1}(\tau)}{\|w\|_{L^p(\T^n)}^{p-1}}$), $p'>1$. By duality between \eqref{eqw} and \eqref{fp12} we have
\[
\int_{\T^n}w(\tau)\rho(\tau)\,dx=\underbrace{\int_{\T^n}w(0)\rho(0)\,dx}_{=0\text{ since }w(0)=0}+(\eps-\eta)\underbrace{\iint_{Q_\tau}[-(-\Delta)^\frac12 u_\eps]\rho\,dxdt}_{I}.
\]
We can write the energy formulation by Remark \ref{energy} and the fact that the boundedness of $Du_\eps$ allows us to consider the nonlocal Hamilton-Jacobi equation as a fractional heat equation driven by $\partial_t\cdot+(-\Delta)^\frac12\cdot $. Hence \[
u\in L^2(0,\tau;H^1(\T^n)),\ \partial_t u\in L^2(0,\tau;L^2(\T^n))\]
by the maximal $L^2$ regularity for fractional heat equations \cite{HieberPruss}. Note now that $W^{1,p}$ coincides with $H_p^1$, where $H_p^1$ is the space of Bessel potentials with norm
\[
\|g\|_{H^{1}_p(\T^n)}=\|(I-\Delta)^\frac12g\|_{L^p(\T^n)}\simeq \|g\|_{L^p(\T^n)}+\|(-\Delta)^\frac12g\|_{L^p(\T^n)},
\]
see \cite[Remark 2.3]{CGsima}. Therefore, we have for $g:\R^n\to\R$ and a constant $A>0$ \cite[Theorem V.3 and Lemma V.3]{Stein}
\begin{equation}\label{semieq}
\|(I-\Delta)^\frac12g\|_{L^p(\T^n)}\leq A\|g\|_{W^{1,p}(\T^n)}.
\end{equation}
We apply the H\"older's inequality on the last integral to find
\begin{align*}
(\eps-\eta)I&\leq (\eps-\eta)\|(-\Delta)^\frac12 u_\eps\|_{L^1(0,\tau;L^p(\T^n))}\|\rho\|_{L^\infty(0,\tau;L^{p'}(\T^n))}\\&\leq A(\eps-\eta)\|u_\eps\|_{L^1(0,\tau;W^{1,p}(\T^n))}\|\rho\|_{L^\infty(0,\tau;L^{p'}(\T^n))},
\end{align*}
where the last inequality follows  from \eqref{semieq}. Since the domain is compact, we can bound \[\|Du_\eps\|_{L^p(Q_T)}\leq C_1(T,p)\|Du_\eps\|_{L^\infty(Q_T)}\] for a constant $C_1>0$, and the right-hand side is bounded by a constant independent of $\eps,\eta$ by Lemma 2.5 in \cite{Biswas} (or even by using the gradient bounds in \cite{CGsima} via the semiconcavity estimates, which are independent of the viscosity). It remains to bound $\|\rho\|_{L^\infty_t(L^{p'}_x)}$. Since $u_\eps,u_\eta$ are $L^1_t(L^\infty_x)$ semiconcave independently of the viscosity (see e.g. \cite[Proposition 3.6]{CGsima}), we have that 
\[
\mathrm{div}(b)\geq -c(t),\ c(\cdot)\in L^1_t,\ c:[0,\tau]\to(0,\infty),
\]
and thus $\|[\mathrm{div}(b)]^-\|_{L^1_tL^\infty_x}<\infty$. Indeed, using that $H\in C^2$ is uniformly convex and $D^2 u_\eps,D^2 u_\eta\leq k(t)\in L^1_t$, $k:[0,T]\to(0,\infty)$, we conclude
\[
\mathrm{div}(b)=-\int_0^1 \sum_{i,j}D^2_{p_ip_j}H(\zeta Du_\eps+(1-\zeta)Du_\eta)\left(\zeta\partial_{x_ix_j}u_\eps+(1-\zeta)\partial_{x_ix_j}u_\eta\right)\,d\zeta\geq -c(t).
\]
To show the previous bound we consider only the term involving $A:=\int_0^1 \left(\zeta D^2_{pp}H\right)$, which verifies $\theta I_n\leq A\leq \Theta I_n$ by the assumptions on $H$, the other being similar. Indeed, we have
\[
\sum_{i,j}A_{ij}\partial_{x_ix_j}u_\eps=\sum_{i,j}A_{ij}(\underbrace{\partial_{x_ix_j}u_\eps-k(t)\delta_{ij}}_{U_\eps}+k(t)\delta_{ij}))=\mathrm{Tr}(A\ U_\eps)+k(t)\sum_i A_{ii}\leq n\Theta k(t),
\]
where we used that $U_\eps\leq 0$ implies $\mathrm{Tr}(A\ U_\eps)\leq0$. We are thus in position to apply Lemma \ref{lemmaFP} with $q=p'$ to find 
\[
\|\rho\|_{L^\infty(0,\tau;L^{p'}(\T^n))}\leq C_F\|\rho(\tau)\|_{L^{p'}(\T^n)}
\]
for a positive constant $C_F$ independent of the viscosity. We then have
\[
\int_{\T^n}w(\tau)\rho(\tau)\,dx\leq C(\eps-\eta)\|\alpha\|_{L^{p'}(\T^n)},
\]
where $C$ does not depend neither on $\eps$ nor $\eta$. By duality, passing to the supremum over $\rho(\tau)=\alpha\in L^{p'}$ we find
\[
\|(u_\eps-u_\eta)(\tau)\|_{L^{p}(\T^n)}\leq C(\eps-\eta).
\]
\end{proof}

\begin{rem}\label{finalrem}
The previous computations (see \eqref{eqw}) suggest the validity of the following one-side bound
\[
\|(u_{\eps,\frac12}-u)^+\|_{L^\infty(Q_T)}\leq C\eps
\]
under the unilateral ($1/2$-semi-superharmonic) assumption
\[
-(-\Delta)^\frac12u_\eps\leq C.
\]
This follows by the same duality argument plainly from \eqref{eqw} and using the following properties of $\rho$:
\[
\rho\geq0\text{ and }\int_{\T^n}\rho(x,t)\,dx=1\text{ for all }t\in[0,\tau].
\]
This weakens the order of the regularity condition required for the local case, i.e. $\Delta u\leq C$ (the so-called semi-superharmo\-nicity, cf. \cite{L82book}), though it changes into a nonlocal requirement. Furthermore, we do not know if Theorem \ref{main} holds up to $p=\infty$, since the equivalence between $H^{1}_p$ and $W^{1,p}$ fails at the endpoint $p=\infty$, see for instance \cite[Theorem 6-(c)]{Taibleson}.
\end{rem}

\begin{rem}[Comparison with the literature]\label{comparison}
The previous estimate partially improves the convergence rate for the (nonlocal) vanishing viscosity process of first-order Hamilton-Jacobi equations. In particular, \cite[Theorem 2.8]{Biswas} and \cite[Theorem 6]{DroniouImbert}  established that for Lipschitz solutions (and Lipschitz initial conditions) and locally Lipschitz Hamiltonians one has
\[
\|u_{\eps,\frac12}-u\|_\infty\leq C\eps|\log\eps|.
\]
In contrast to \cite{Biswas,DroniouImbert}, we do not exploit the explicit representation of the half-Laplacian, but use the equivalence $H^1_p\equiv W^{1,p}$, $p\in(1,\infty)$. The results of \cite{Biswas,DroniouImbert} depend on appropriate scaling and decay properties of the kernel of the nonlocal operator.\\
The importance of our findings also relies on the two-side $\mathcal{O}(\eps)$ rate in $L^p$ norms, $p\in(1,\infty)$, for semiconcave solutions, since the regularization by the Laplacian yields the rate
\begin{equation*}
\|u_{\eps,1}-u\|_{L^\infty(0,T;L^p(\T^n))}\leq C\eps^{\frac{1}{2}+\frac{1}{2p}},
\end{equation*}
which deteriorates as $p\to\infty$, see \cite{CGM}. It is worth mentioning that the recent paper \cite[Proposition 4.4]{Tranetal} highlights some particular one-dimensional cases of Hamiltonians $H$ and/or initial data $u_0$ for which one has a two-side estimate in $L^\infty$ norm with rate $\mathcal{O}(\eps)$ or $\mathcal{O}(\eps|\log\eps|)$ when the viscosity is $\eps\Delta u$. 
\end{rem}

\begin{rem}
During the proof we used the compactness of the state space to bound $\|Du_\eps\|_{L^p}$. We expect that the strategy can be repeated on the whole space $\R^n$ asking only $u$ to be $L^1_t(L^\infty_x)$ semiconcave and using merely a $W^{1,p}$ estimate, $p<\infty$. 
\end{rem}

\vspace{1cm}


\end{document}